\documentclass[a4paper,oneside,10pt,oldfontcommands]{article}
\usepackage[a4paper, total={426pt, 674pt}]{geometry}
\newcommand{\langue}{anglais}	% langue de rédaction de l'article

\usepackage{ifthen}
\ifthenelse{\equal{\langue}{francais}}{
	\usepackage[french]{babel}}
	{\ifthenelse{\equal{\langue}{anglais}}{\usepackage[english]{babel}}{}}
\usepackage[T1]{fontenc}
\usepackage[utf8]{inputenc}
\usepackage[babel,german=guillemets]{csquotes}
\usepackage{eso-pic}

\usepackage[backend=biber, style=alphabetic, sorting=nyt, giveninits=true, doi=false, isbn=false, url=false, eprint=false]{biblatex} %Imports biblatex package
\DeclareNameAlias{author}{last-first}
\usepackage{lmodern} %Type1-font for non-english texts and characters
\usepackage{enumitem}
\usepackage{ifthen}
\ifthenelse{\equal{\langue}{francais}}{
	\frenchbsetup{StandardLists=true}}
	{}
	
\DeclareUnicodeCharacter{00A0}{~}

\usepackage{authblk}

\usepackage{graphicx} %%For loading graphic files
\usepackage{amsmath}
\usepackage{amssymb}
\usepackage{amsthm}
\usepackage{amsfonts}
\usepackage{bbold}
\usepackage{mathtools}
\usepackage{stmaryrd}
\usepackage{tikz-cd}
\usepackage{hyperref}
\usepackage{multirow}

\ifthenelse{\equal{\langue}{francais}}{
	\newcommand{\theoremenom}{Théorème}
	\newcommand{\propositionnom}{Proposition}
	\newcommand{\lemmenom}{Lemme}
	\newcommand{\corollairenom}{Corollaire}
	\newcommand{\definitionnom}{Définition}
	\newcommand{\remarquenom}{Remarque}
	\newcommand{\exemplenom}{Exemple}
	\newcommand{\conjecturenom}{Conjecture}
}{\ifthenelse{\equal{\langue}{anglais}}{
	\newcommand{\theoremenom}{Theorem}
	\newcommand{\propositionnom}{Proposition}
	\newcommand{\lemmenom}{Lemma}
	\newcommand{\corollairenom}{Corollary}
	\newcommand{\definitionnom}{Definition}
	\newcommand{\remarquenom}{Remark}
	\newcommand{\exemplenom}{Example}
	\newcommand{\conjecturenom}{Conjecture}
}{}}
	
\newtheorem{theoreme}{\theoremenom}
\newtheorem{proposition}[theoreme]{\propositionnom}
\newtheorem{lemme}[theoreme]{\lemmenom}
\newtheorem{corollaire}[theoreme]{\corollairenom}
\newtheorem{definition}[theoreme]{\definitionnom}
\newtheorem*{remarque}{\remarquenom}

\newtheorem*{conjecture}{\conjecturenom}

\usepackage{thmtools}
\usepackage{thm-restate}

\makeatletter
\def\cleartheorem#1{%
    \expandafter\let\csname#1\endcsname\relax
    \expandafter\let\csname c@#1\endcsname\relax
}
\makeatother
\newcommand{\compteurThm}{1}
\newcounter{annexe}

\newcommand{\C}{\mathbb{C}}
\newcommand{\R}{\mathbb{R}}
\newcommand{\N}{\mathbb{N}}

\newcommand{\supp}{\text{supp}}

\begin{document}

\pagestyle{empty} %No headings for the first pages.

%% Title Page %%%%%%%%%%%%%%%%%%%%%%%%%%%%%%%%%%%%%%%%%%%%%%%
%% ==> Write your text here or include other files.

\renewcommand*{\thefootnote}{\fnsymbol{footnote}}
%% The simple version:
\title{Sufficient conditions yielding the Rayleigh Conjecture for the clamped plate}
\author{
Roméo Leylekian
\footnote{Aix-Marseille Université, CNRS, I2M, Marseille, France - \textbf{email:} romeo.leylekian@univ-amu.fr}
}
\date{} %%If commented, the current date is used.
\maketitle

%% The nice version:
%\input{page_de_garde} %%You need a file 'page_de_garde.tex' for this.
%% ==> TeXnicCenter supplies a possible titlepage file
%% ==> with its templates (File | New from Template...).

%% Abstract
\begin{abstract}
The Rayleigh Conjecture for the bilaplacian consists in showing that the clamped plate with least principal eigenvalue is the ball. The conjecture has been shown to hold in 1995 by Nadirashvili in dimension $2$ and by Ashbaugh and Benguria in dimension $3$. Since then, the conjecture remains open in dimension $d\geq 4$. In this paper, we contribute to answer this question, and show that the conjecture is true in any dimension as long as some special condition holds on the principal eigenfunction of an optimal shape. This condition regards the mean value of the eigenfunction, asking it to be in some sense minimal.
This main result is based on an order reduction principle allowing to convert the initial fourth order linear problem into a second order affine problem, for which the classical machinery of shape optimization and elliptic theory is available.
The order reduction principle turns out to be a general tool. In particular, it is used to derive another sufficient condition for the conjecture to hold, which is a second main result. This condition requires the Laplacian of the optimal eigenfunction to have constant normal derivative on the boundary.
Besides our main two results, we detail shape derivation tools allowing to prove simplicity for the principal eigenvalue of an optimal shape and to derive optimality conditions. Finally, because our first result involves the principal eigenfunction of a ball, we are led to compute it explicitly. 
\end{abstract}

{
}
%% Remerciements
%\section*{Remerciements}

%\thispagestyle{empty}
%\clearpage

%% Inhaltsverzeichnis %%%%%%%%%%%%%%%%%%%%%%%%%%%%%%%%%%%%%%%
%\tableofcontents %Table of contents
%\cleardoublepage %The first chapter should start on an odd page.

\pagestyle{plain} %Now display headings: headings / fancy / ...

%% Chapters %%%%%%%%%%%%%%%%%%%%%%%%%%%%%%%%%%%%%%%%%%%%%%%%%
%% ==> Write your text here or include other files.

\section{Introduction}

In 1877, at the same time he was formulating his famous conjecture regarding fixed membranes, Rayleigh stated that the principal frequency of a clamped plate should be minimal when the plate is circular. Let us explain more precisely the terms of this claim. The principal frequency of a clamped plate involves the eigenvalue problem related to the bilaplacian with Dirichlet boundary conditions (also refered to as Dirichlet bilaplacian), which is the following eigenvalue problem.
\begin{equation}\label{eq:equation aux vp}
\left\{\begin{array}{rcll}
\Delta^2 u & = & \Gamma u & in \quad\Omega, \\
u & = & 0 & on \quad\Omega,\\
\partial_n u & = & 0 & on \quad\Omega.
\end{array}\right.
\end{equation}
Here $\Omega\subseteq\R^d$ ($d\in\N^*$) stands for an arbitrary bounded open set, $u\in H_0^2(\Omega)$, $\Gamma$ is a real number, and $\partial_n=\vec{n}\cdot\nabla$ is the partial derivative in the direction of the outward normal unit vector $\vec{n}$. It turns out that problem (\ref{eq:equation aux vp}) admits countably many (nontrivial) eigencouples $(u,\Gamma)$, and that the sequence of eigenvalues is positive and grows up to infinity. This occurs since the resolvent of the Dirichlet bilaplacian is compact positive self-adjoint when seen as an operator acting on $L^2(\Omega)$ (see \cite{gazzola-grunau-sweers} for a collection of general facts regarding the bilaplacian and, more generally, polyharmonic operators). The principal eigenvalue of the clamped plate is nothing else but the lowest of these eigenvalues, that we will denote $\Gamma(\Omega)$ in the rest of the document in order to emphasize its dependance on the open set $\Omega$. As for any eigenvalue of a self-adjoint operator, $\Gamma(\Omega)$ admits a variational characterization, which is the following:

\begin{equation}\label{eq:formulation variationnelle}
\Gamma(\Omega)=\min_{\substack{u\in H_0^2(\Omega)\\ u\neq 0}}\frac{\int_\Omega(\Delta u)^2}{\int_\Omega u^2}.
\end{equation}

Initially stated in the context of subsets of $\R^2$ only, the Rayleigh Conjecture deals with the problem of determining the open set with least principal eigenvalue among all open sets having same measure. As its counterpart for the Dirichlet Laplacian, the conjecture claims that such a set exists, is \enquote{almost} unique, and is given by the Euclidean ball fitting the volume constraint. Note that plain uniqueness does not hold since $\Gamma(\Omega)$ is invariant under isometries of $\Omega$ and under removing a set of zero $H^2$-capacity from $\Omega$ (see sections 3.3 and 3.8.1 of \cite{henrot-pierre} for the definition of capacity). In other words, if $|.|$ denotes the $d$-dimensional Lebesgue measure,

\begin{conjecture}
Let $\Omega$ be a bounded open subset of\/ $\R^d$ and $B$ a ball such that $|B|=|\Omega|$. Then,
\begin{equation}\label{eq:conjecture}
\Gamma(\Omega)\geq\Gamma(B).
\end{equation}
Moreover there is equality if and only if\/ $\Omega$ is a ball (up to a set of zero $H^2$-capacity).
\end{conjecture}

After its publication in 1877, one of the first serious results on the conjecture is due to Szegö \cite{szego}, and states, based on symmetrisation arguments, that, as soon as the eigenfunction associated with the first eigenvalue on a set $\Omega$ is of fixed sign, the Faber-Krahn type inequality (\ref{eq:conjecture}) holds. However, one of the main challenges when working with fourth and higher order elliptic operators is the vacuity of the maximum principle in general for arbitrary domains. This means that, unlike the Dirichlet Laplacian, the one-sign property of the principal eigenfunction is no longer guaranteed as a consequence of the non-applicability of Krein-Rutmann Theorem. Indeed, the first - and maybe the most famous - example of domains in which this one-sign property fails was found to be annuli with small inner radius in 1952 \cite{duffin-shaffer, coffman-duffin-shaffer}. On the contrary, balls do enjoy the one-sign property (see Proposition \ref{prop:fonction propre boule}, in which the principal eigenvalue of a ball and the associated eigenfunction are computed). This situation is troublesome in the sense that, at first glance, it deprives us of our principal tool in shape optimization, which is symmetrisation.

Nevertheless, using perturbation techniques, Mohr \cite{mohr} showed in 1975 that any planar optimal regular shape, if it exists, has to be the ball. Such strategies, based on optimality conditions, are common in shape optimization. They have proved to work also for the buckling problem \cite{willms-weinberger}. In both clamped and buckling  problems, for distinct reasons, the approach strongly relies on the planeness of the shapes involved. The result of Mohr was finally outshined by a series of papers beginning with \cite{talenti76} in 1976, in which Talenti proved its famous comparison principle. An astute adaptation of this principle allowed him to find in 1981 a lower bound on the optimal eigenvalue depending on the dimension (see \cite{talenti81}). Following this strategy, Nadirashvili solved the conjecture in $\R^2$ in 1995 in \cite{nadirashvili}. Subsequently, still in the wake of Talenti's approach, Ashbaugh and Benguria proved the conjecture in $\R^2$ and $\R^3$ in 1995 (see \cite{ashbaugh-benguria}). Furthermore, in 1996, Ashbaugh and Laugesen \cite{ashbaugh-laugesen} completely solved Talenti's \enquote{two-ball problem} (see \cite[equation (26)]{ashbaugh-benguria} for details) in any dimension. As a result, they showed on the one hand that the plain approach of Talenti could not answer the Rayleigh Conjecture when $d\geq4$, but, on the other hand, gave a very precise lower bound on the optimal eigenvalue. Since then, up to our knowledge, no significant breakthrough has been performed regarding the actual optimal shape nor the actual optimal eigenvalue in high dimension. Let us however mention our work \cite{leylekian}, in which we obtain a surprising sufficient condition for the Rayleigh Conjecture to hold, based on a refinement of Talenti's approach. As a final word, we cite the interesting papers of Krist\'{a}ly \cite{kristaly20,kristaly22} dealing with the conjecture in non-Euclidean setting.

\newcommand{\reg}{$C^4$}
The goal of the present document is to contribute for a better understanding of the terms of validity of the Rayleigh Conjecture. More precisely, under existence and regularity of an optimal shape, we will show that the conjecture is true in any dimension whenever the principal eigenfunction satisfies some special condition. This will be explained in the next lines. First, we need to assume that there exists a solution with \textbf{\reg~regular connected boundary} to the problem
\begin{equation}\label{eq:pb}
\min\{\Gamma(\Omega):\Omega\subseteq\R^d\text{ bounded open set, }|\Omega|=c\},
\end{equation}
where $c$ is a fixed positive real number. Here, we recall that the question of the existence of an optimal shape is still open (see however the recent work \cite{stollenwerk} dealing with this issue for domains contained in a given large box). In the rest of the document, we will denote $\Omega$ a $C^4$ regular solution to (\ref{eq:pb}). The regularity assumption on $\Omega$ will be used for invoking shape derivation. Indeed, it guarantees that the eigenfunctions are $H^4(\Omega)$ (see \cite[Theorem 2.20]{gazzola-grunau-sweers}). However, besides $H^4$ regularity, at some point we will need more regularity for the principal eigenfunction. The $L^p$ regularity theory (see again \cite[Theorem 2.20]{gazzola-grunau-sweers}) will answer this need by providing $W^{4,p}(\Omega)$ regularity, and then (thanks to Sobolev emebddings) $C^{3,\alpha}(\overline{\Omega})$ regularity for the eigenfunction. On the other hand, the assumption on the geometry of the boundary is technical as we shall see in the proof of our main theorem. We stress the properties of regularity and geometry enjoyed by $\Omega$ by stating the assumption
\begin{equation}\label{hyp:rg}
\tag{RG}
\Omega\textit{ is \reg and\/ }\partial\Omega \textit{ is connected}.
\end{equation}

Apart from (\ref{hyp:rg}), we will need another special assumption to run our proof. This condition asks for the mean value $|\int_\Omega u|$ of the first $L^2$-normalised eigenfunction $u$ in $\Omega$ to be minimal. Then, the main conclusion of the present document is the theorem stated below.

\begin{theoreme}\label{thm:faber-krahn moyenne}
Let $\Omega$ be an optimal shape for problem (\ref{eq:pb}) satisfying (\ref{hyp:rg}) and $B$ a ball such that $|\Omega|=|B|$. Let $u$ be a first $L^2$-normalised eigenfunction in $\Omega$ and $u_B$ a first $L^2$-normalised eigenfunction in $B$. Then,
\begin{equation}\label{hyp:moyennes}\tag{M}
\left|\int_\Omega u\right|\geq\left|\int_B u_B\right|.
\end{equation}
Moreover, (\ref{hyp:moyennes}) holds with equality if and only if\/ $\Omega=B$ (up to a translation).
\end{theoreme}

\begin{remarque}
Roughly speaking, Theorem \ref{thm:faber-krahn moyenne} tells that an optimal shape of which the mean of the principal eigenfunction is minimal is a ball. Therefore, one is led to wonder if the minimality of the $H_0^2$ norm of an eigenfunction implies the minimality of its mean. Among others, this question will be addressed in section \ref{sec:corollaires}.
\end{remarque}

The proof of Theorem \ref{thm:faber-krahn moyenne} is based on a procedure that we shall call \enquote{order reduction principle}. Such a procedure appears to be new, at least in the present context, although ensuing from recurrent ideas (see for instance \cite[section 1.1.3]{gazzola-grunau-sweers} and \cite[equation (2.3)]{antunes-buoso-freitas}). In essence, the order reduction principle allows to turn the fourth order eigenvalue problem (\ref{eq:equation aux vp}) into a second order affine problem, for which a more sophisticated machinery is available. In particular, it becomes possible to use symmetrisation techniques, which are the other main ingredient for proving Theorem \ref{thm:faber-krahn moyenne}. However, we would like to emphasize that the order reduction principle paves the way for the utilization of many other tools coming from the field of second order elliptic operators. In order to illustrate this fact, we derive another main result, which is based on the theory of overdetermined problems stemming from the historical \cite{serrin}. Before, let us simply recall that very little is known in general on overdetermined problems of fourth order, \cite{bennett,payne-schaefer,dalmasso,  barkatou} being the almost exhaustive list of results.

\begin{theoreme}\label{thm:surdetermine}
Let $\Omega$ be an optimal shape for problem (\ref{eq:pb}) satisfying (\ref{hyp:rg}). Let $u$ be a first eigenfunction on $\Omega$ such that $\partial_n\Delta u$ is constant on $\partial\Omega$. Then, $\Omega$ is a ball.
\end{theoreme}

Actually, the proofs of Theorem \ref{thm:faber-krahn moyenne} and Theorem \ref{thm:surdetermine} do not appeal to the order reduction principle as a standalone. Indeed, to reveal its potential, the order reduction principle needs to thrive on the optimality condition satisfied by an optimal shape $\Omega$. Such an optimality condition shall be derived only when the eigenvalue $\Gamma(\Omega)$ is simple. Even if the question of simplicity of the optimal eigenvalue had already been tackled in \cite{mohr}, one of the main results of the present work is to propose a thorough proof of this fact and to derive the subsequent optimality condition, which is precised in the next theorem.

\begin{restatable}{retheoreme}{simplicite}\label{thm:condition d'optimalité}
Let $\Omega$ be a \reg\/ open set solving (\ref{eq:pb}). Then, $\Gamma(\Omega)$ is simple. Moreover, if $u$ denotes an $L^2$-normalised eigenfunction associated with $\Gamma(\Omega)$, $\Delta u$ is a.e. constant equal to $\pm\alpha$ on any connected component of $\partial\Omega$, where
\begin{equation*}
\alpha:=\sqrt{\frac{4\Gamma(\Omega)}{d|\Omega|}}.
\end{equation*}
\end{restatable}

In the remainder of this document we will detail the proofs of Theorem \ref{thm:faber-krahn moyenne}, Theorem \ref{thm:surdetermine} and Theorem \ref{thm:condition d'optimalité}. In section \ref{sec:order reduction}, we present our main tool, which is the order reduction principle, roughly explained in the previous lines. Section \ref{sec:derivation} gathers some results about derivation of simple and multiple eigenvalues of the Dirichlet bilaplacian. Using these tools, in section \ref{sec:simplicity}, we prove Theorem~\ref{thm:condition d'optimalité}. Section \ref{sec:theorem} is devoted to the proofs of Theorem~\ref{thm:faber-krahn moyenne} and Theorem~\ref{thm:surdetermine}. Section \ref{sec:corollaires} discusses two consequences of Theorem \ref{thm:faber-krahn moyenne}.

\section{Order reduction principle}\label{sec:order reduction}

The order reduction principle, from which arise Theorem \ref{thm:faber-krahn moyenne} and Theorem \ref{thm:surdetermine}, is an algebraic trick leading to an \enquote{eigenvalue problem} involving a differential operator of order lower than the bilaplacian, that is, the Laplacian. The counterpart to the reduction of the order is that the \enquote{eigenvalue problem} is not linear anymore. The precise statement is encapsulated in the next proposition.

\begin{proposition}\label{prop:principe de réduction d'ordre}
Let $\Omega$ be a $C^4$ bounded open set, and $u\in H_0^2(\Omega)$ an eigenfunction of the bilaplacian in $\Omega$ associated with an eigenvalue $\mu$, so that $\Delta u$ has trace in $H^{\frac{3}{2}}(\partial\Omega)$. Finally, let $g_u$ satisfy
$$
\left\{
\begin{array}{rcll}
\Delta g_u & = & 0 & in \quad\Omega,\\
g_u & = & \frac{\Delta}{\sqrt{\mu}}u & on \quad\partial\Omega.\\
\end{array}
\right.
$$
Then, the function $z_u:=\frac{\Delta}{\sqrt{\mu}}u+u-g_u$ solves the equation
\begin{equation}\label{eq:pb simplifie}
\left\{
\begin{array}{rcll}
\Delta z_u & = & \sqrt{\mu}(z_u + g_u) & in \quad\Omega,\\
z_u & = & 0 & on \quad\partial\Omega.\\
\end{array}
\right.
\end{equation}
In particular, $z_u$ solves the following problem, the value of which is $\frac{1}{\sqrt{\mu}}$:
\begin{equation}\label{eq:formule variationnelle zu}
\frac{1}{\sqrt{\mu}}=-\min_{\substack{z\in H_0^1(\Omega) \\ z\neq0}}\frac{\int_\Omega z^2+\int_\Omega g_u(2z-z_u)}{\int_\Omega|\nabla z|^2}
\end{equation}
Moreover, if $g_u\geq0$, then $z_u< 0$.
\end{proposition}

\begin{proof}
The eigenfunction $u$ satisfies by definition
\begin{equation*}\label{eq:pb vp}
\left\{
\begin{array}{rcll}
(\Delta^2-\mu) u & = & 0 & in \quad\Omega,\\
u & = & 0 & on \quad\partial\Omega,\\
\partial_n u & = & 0 & on \quad\partial\Omega.
\end{array}
\right.
\end{equation*}
The idea now relies on observing that $(\Delta^2-\mu)=(\Delta -\sqrt{\mu})(\Delta+\sqrt{\mu})$. Hence, setting $y=\left(\frac{\Delta}{\sqrt{\mu}}+1\right) u$, $y$ verifies $\Delta y = \sqrt{\mu}y$ in $\Omega$. Nevertheless, the boundary condition for $y$ is $y=\frac{\Delta}{\sqrt{\mu}} u$ on $\partial\Omega$. Note that $\frac{\Delta}{\sqrt{\mu}} u\in H^{\frac{3}{2}}(\partial\Omega)$ since $\Delta u\in H^2(\Omega)$ thanks to the regularity assumption made on $\partial\Omega$ (see \cite[Theorem 2.20]{gazzola-grunau-sweers}). But if $g_u$ is the solution to the Dirichlet problem $\Delta g_u=0$ in $\Omega$ and $g_u=\frac{\Delta}{\sqrt{\mu}} u$ on the boundary, setting $z_u:=y-g_u=\frac{\Delta}{\sqrt{\mu}} u + u-g_u$, one gets that $z_u$ is an $H_0^1(\Omega)\cap H^2(\Omega)$ function satisfying
\begin{equation*}\label{eq:pb simplifie naif}
\Delta z_u = \sqrt{\mu}(z_u + g_u).
\end{equation*}
In particular $z_u$ is a critical point of the functional $E_\mu$ defined on $H_0^1(\Omega)$ and given by
$$
E_\mu(z)=\int_\Omega|\nabla z|^2+\sqrt{\mu}\int_\Omega z^2+2\sqrt{\mu}\int_\Omega g_uz.
$$
Moreover, $E_\mu$ being strictly convex, $z_u$ is the unique minimiser. But, from the equation involving $z_u$, we derive the identity $E_\mu(z_u)=\sqrt{\mu}\int g_uz_u$. In this context, the relation
$$
\int_\Omega|\nabla z|^2+\sqrt{\mu}\int_\Omega z^2+2\sqrt{\mu}\int_\Omega g_uz\geq\sqrt{\mu}\int_\Omega g_uz_u,
$$
holding for all $z\in H_0^1(\Omega)$, is an equality if and only if $z=z_u$. Moreover, thanks to elementary manipulations, this inequality can be turned into the next one, which, as before, is attained if and only if $z=z_u$.
$$
-\frac{\int_\Omega z^2+\int_\Omega g_u(2z-z_u)}{\int_\Omega|\nabla z|^2}\leq\frac{1}{\sqrt{\mu}}.
$$
This completes the proof of (\ref{eq:formule variationnelle zu}). Finally, if $g_u\geq0$, the strong maximum principle applied to the operator $\Delta-\sqrt{\mu}$ in (\ref{eq:pb simplifie}) shows that $z_u<0$ unless $z_u$ vanishes identically in $\Omega$. But if $z_u=0$, due to (\ref{eq:pb simplifie}), $g_u=0$, and in turn $-\Delta u=\sqrt{\mu} u$ in $\Omega$. As a result, the function $v:\Omega\times\R\to\R$ defined by $v(x,t):=u(x)e^{\sqrt[4]{\mu}t}$ is harmonic in $\Omega\times\R$ and satisfies $v=\partial_n v=0$ on $\partial(\Omega\times\R)$. Thanks to \cite{tolsa}, we conclude that $v$ vanishes identically in each connected component of $\Omega\times\R$. Hence $u$ vanishes identically in $\Omega$.
\iffalse
But since $u$ does not vanish identically in $\Omega$, by analyticity $\partial\Omega$ does not belong to the closure of the nodal set $\{u=0\}\cap\Omega$. Therefore, one may find a point $p\in\partial\Omega$ at distance $\delta>0$ from the nodal set. Up to a sign change of $u$, we assume that $p\in\partial\{u>0\}$. Now recall that $\Omega$, being $C^4$, satisfies an interior ball condition at $p$. Shrinking the corresponding ball for its diameter to be less that $\delta$, we obtain that $\{u>0\}$ satisfies an interior ball condition at $p$. Then, Hopf boundary Lemma applies at $p$ yielding $\partial_n u(p)\neq 0$, which is a contradiction since $u\in H_0^2(\Omega)$. Therefore, we conclude that $z_u<0$.
\fi
\end{proof}

\begin{remarque}
\begin{enumerate}
\item Setting $y=\left(\frac{\Delta}{\sqrt{\mu}}-1\right) u$ instead of $y=\left(\frac{\Delta}{\sqrt{\mu}}+1\right) u$, we see that the function $z_u':=\frac{\Delta}{\sqrt{\mu}} u-u-g_u$ is $H_0^1(\Omega)\cap H^2(\Omega)$ and satisfies $-\Delta z_u'=\sqrt{\mu}(z_u'+g_u)$. However, we cannot obtain a variational formulation similar to (\ref{eq:formule variationnelle zu}) involving $z_u'$ since, unlike $E_\mu$, the energy functional of which $z_u'$ is a critical point is not convex.
\item Note that the system (\ref{eq:pb simplifie}) is linear with respect to $(z_u,g_u)$. As a consequence, the variational formula (\ref{eq:formule variationnelle zu}) remains true when replacing $z_u$ and $g_u$ respectively with $\gamma z_u$ and $\gamma g_u$ for any $\gamma\in\R\setminus\{0\}$.
\item The regularity on $\Omega$ can be weakened in some cases. More precisely, one shall run the proof of (\ref{eq:pb simplifie}) and (\ref{eq:formule variationnelle zu}) as long as $\Delta u\in H^1(\Omega)$.
\end{enumerate}
\end{remarque}

Surprisingly, Proposition \ref{prop:principe de réduction d'ordre} will not only serve proving Theorem \ref{thm:faber-krahn moyenne} and Theorem \ref{thm:surdetermine}. Indeed, it has the following consequence which will be very useful to prove the simplicity of the optimal first eigenvalue. First, let us recall that, in the case of fourth order equations, it is not known whether having $u=\partial_n u=\partial_n^2 u=0$ on some arbitrary portion $\gamma$ of $\partial\Omega$ yields $u=0$ in the neighbourhood of $\gamma$. The lack of this property (called uniqueness continuation), is due to the fact that neither Hölmgren principle nor Hopf boundary Lemma apply in this framework (see however Theorem 1.1 of \cite{ortega-zuazua} and the discussion above and below its statement).

\begin{corollaire}\label{cor:principe de réduction d'ordre}
Let $\Omega$ be a $C^4$ bounded open set, and $u\in H_0^2(\Omega)$ satisfy $\Delta^2u=\mu u$ for some $\mu>0$. Assume that $\Delta u=0$ on $\partial\Omega$. Then, $u=0$ in $\Omega$.
\end{corollaire}

\begin{proof}
Assume that $u$ does not vanish identically, so that it is an eigenfunction. The hypothesis $\Delta u=0$ on $\partial\Omega$ reads $g_u=0$ on $\partial\Omega$ and then in $\Omega$, where $g_u$ is defined as in Proposition \ref{prop:principe de réduction d'ordre}. Then, the function $z_u$ satisfies $\Delta z_u=\sqrt{\mu}z_u$. This means that either $z_u=0$, or $-\sqrt{\mu}$ is an eigenvalue of the Dirichlet Laplacian. As the latter cannot hold, $z_u=0$, and hence $-\Delta u=\sqrt{\mu} u$, so that $u$ is an eigenfunction of the Dirichlet Laplacian. Because $\partial_n u=0$, we run into a contradiction applying \cite{tolsa} to the harmonic extension of $u$ as in the end of the proof of Proposition \ref{prop:principe de réduction d'ordre}.
\end{proof}

\begin{remarque}
Corollary \ref{cor:principe de réduction d'ordre} holds under weaker regularity assumptions on $\Omega$. For instance, it is enough that\/ $\Omega$ is Lipschitz with small constant (see \cite{tolsa}), and satisfies a uniform outer ball condition. Indeed, under the outer ball condition and the Lipschitz regularity assumption, the system solution of the equation $\Delta^2 v=f$ on $\Omega$, $v=\Delta v=0$ on $\partial\Omega$ (see \cite[Example 2.33]{gazzola-grunau-sweers} for the definition of system and energy solutions) is $H_0^1\cap H^2(\Omega)$ due to \cite[Theorem 1.1]{adolfsson}. Thus it coincides with the energy solution. In particular, in Corollary \ref{cor:principe de réduction d'ordre}, $u$ being an energy solution, we get that $\Delta u\in H_0^1(\Omega)$. Therefore, Proposition \ref{prop:principe de réduction d'ordre} still applies (see the remark under its proof) and the proof of Corollary \ref{cor:principe de réduction d'ordre} works as long as the Lipschitz constant of $\Omega$ is small enough for using \cite{tolsa}.
\end{remarque}

\section{Shape derivatives}\label{sec:derivation}

In order to fully exploit Proposition \ref{prop:principe de réduction d'ordre}, one needs to gain information on the function $g_u$ (defined in the statement of the Proposition \ref{prop:principe de réduction d'ordre}) when $\Omega$ is an optimal shape. As $g_u$ depends on the value of $\Delta u$ on $\partial\Omega$, one might use shape derivatives. Shape derivatives for eigenvalues of polyharmonic operators are less famous than their counterparts for the Laplacian, for which one might refer to the classical textbook \cite{henrot-pierre}. Note moreover that this reference does not deal in details with the derivative of multiple eigenvalues. For a framework on the derivation of simple and multiple eigenvalues of a general abstract operator see \cite{haug-rousselet, lamberti-cristoforis}. For the concrete shape derivation of simple and multiple eigenvalues of the bilaplacian and polyharmonic operators, we found only few references \cite{buoso-lamberti13,buoso-lamberti15,ortega-zuazua,buoso,antunes-buoso-freitas}. In this section, we shall refer to \cite{buoso-lamberti15}, in which results on derivatives of multiple eigenvalues of several operators including the Dirichlet bilaplacian are obtained. For that purpose, assume $\Omega$ to be arbitrary, and let $\Gamma_k^\Omega$ be the functional defined on $C^2(\R^d,\R^d)$ by
\begin{equation}\label{eq:Gamma_Omega}
\Gamma_k^\Omega(V)=\Gamma_k((\text{id}+V)\Omega).
\end{equation}
Here, $\Gamma_k(\Omega)$ denotes the $k$-th eigenvalue of the bilaplacian on $\Omega$, \textbf{counted with multiplicity}. Then, if $\Gamma_k(\Omega)$ is of multiplicity $p\in\N^*$, and if $\Gamma_k(\Omega)=...=\Gamma_{k+p-1}(\Omega)$, \cite{buoso-lamberti15} explains that, in a neighbourhood $\mathcal{W}$ of $0$ in $C^2(\R^d,\R^d)$, the set $\{\Gamma_{k+i-1}^\Omega(V):1\leq i\leq p, V\in\mathcal{W}\}$ is made of the union of $p$ analytic branches (see Figure \ref{fig:vp multiple}). Moreover, the derivatives of these branches at $0$ correspond to the eigenvalues of an explicit matrix, as stated in the next theorem.

\begin{figure}[h!]
\centering
\resizebox{130mm}{!}{%% Creator: Inkscape 1.0.1 (c497b03c, 2020-09-10), www.inkscape.org
%% Accompanies image file 'vp-multiple.pdf' (pdf, eps, ps)
\begingroup%
  \makeatletter%
  \providecommand\color[2][]{%
    \errmessage{(Inkscape) Color is used for the text in Inkscape, but the package 'color.sty' is not loaded}%
    \renewcommand\color[2][]{}%
  }%
  \providecommand\transparent[1]{%
    \errmessage{(Inkscape) Transparency is used (non-zero) for the text in Inkscape, but the package 'transparent.sty' is not loaded}%
    \renewcommand\transparent[1]{}%
  }%
  \providecommand\rotatebox[2]{#2}%
  \newcommand*\fsize{\dimexpr\f@size pt\relax}%
  \newcommand*\lineheight[1]{\fontsize{\fsize}{#1\fsize}\selectfont}%
  \ifx\svgwidth\undefined%
    \setlength{\unitlength}{465.6178251bp}%
    \ifx\svgscale\undefined%
      \relax%
    \else%
      \setlength{\unitlength}{\unitlength * \real{\svgscale}}%
    \fi%
  \else%
    \setlength{\unitlength}{\svgwidth}%
  \fi%
  \global\let\svgwidth\undefined%
  \global\let\svgscale\undefined%
  \makeatother%
  \begin{picture}(1,0.51759831)%
    \lineheight{1}%
    \setlength\tabcolsep{0pt}%
    \put(0,0){\includegraphics[width=\unitlength,page=1]{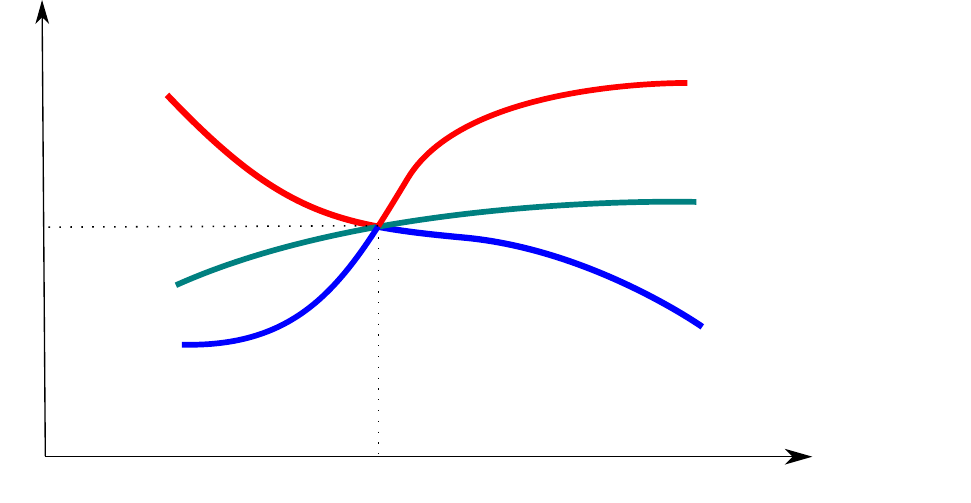}}%
    \put(0.81419531,0.00757143){\color[rgb]{0,0,0}\makebox(0,0)[lt]{\lineheight{1.25}\smash{\begin{tabular}[t]{l}$t$\end{tabular}}}}%
    \put(0.38461087,0.01083797){\color[rgb]{0,0,0}\makebox(0,0)[lt]{\lineheight{1.25}\smash{\begin{tabular}[t]{l}$0$\end{tabular}}}}%
    \put(-0.00325089,0.2772711){\color[rgb]{0,0,0}\makebox(0,0)[lt]{\lineheight{1.25}\smash{\begin{tabular}[t]{l}$\Gamma$\end{tabular}}}}%
    \put(0,0){\includegraphics[width=\unitlength,page=2]{vp-multiple-fig.pdf}}%
    \put(0.26552302,0.09681316){\color[rgb]{0,0,0}\makebox(0,0)[lt]{\lineheight{1.25}\smash{\begin{tabular}[t]{l}$\partial^-$\end{tabular}}}}%
    \put(0,0){\includegraphics[width=\unitlength,page=3]{vp-multiple-fig.pdf}}%
    \put(0.72105257,0.21764111){\color[rgb]{0,0,0}\makebox(0,0)[lt]{\lineheight{1.25}\smash{\begin{tabular}[t]{l}$\partial^+$\end{tabular}}}}%
  \end{picture}%
\endgroup%
}
\caption{Analytic branches near a multiple eigenvalue $\Gamma$ on domains of the form $(\text{id}+tV)\Omega$ for a given set $\Omega$ and vector field $V$. The blue, the green, and the red lines are respectively the graphs of $t\mapsto\Gamma_k^\Omega(tV)$, $t\mapsto\Gamma_{k+1}^\Omega(tV)$, and $t\mapsto\Gamma_{k+2}^\Omega(tV)$. The segment $\partial^-$ represents the tangent generated by the left partial derivative of $\Gamma_k^\Omega$ at $0$ in the direction of $V$. The segment $\partial^+$ represents the tangent generated by the right partial derivative of $\Gamma_k^\Omega$.}
\label{fig:vp multiple}
\end{figure}

\begin{theoreme}\label{thm:derivee vp}
Let $\Omega$ be a $C^4$ bounded open set and $k,p\in\N^*$. Assume that\/ $\Gamma_k(\Omega)=...=\Gamma_{k+p-1}(\Omega)=:\Gamma$ and that $\Gamma$ is of multiplicity $p$. Then, the functionals $\Gamma_k^\Omega,...,\Gamma_{k+p-1}^\Omega$ defined in (\ref{eq:Gamma_Omega}) are Gâteau-differentiable at $0$ both on the right and on the left, and their partial derivatives in the direction of a vector field $V\in C^2(\R^d,\R^d)$ (both on the right and on the left) shall be mapped in a bijective way to the eigenvalues (counted with multiplicity) of the matrix
\begin{equation}\label{eq:vp multiple}
M_V:=\left(-\int_{\partial\Omega}\Delta u_{k+i-1}\Delta u_{k+j-1} V\cdot\vec{n} \right)_{1\leq i,j\leq p},
\end{equation}
where $u_k,...,u_{k+p-1}$ is any $L^2$-orthonormal basis of the eigenspace corresponding to $\Gamma$.
\end{theoreme}

\begin{proof}
Apply \cite[Theorem 3.2]{buoso-lamberti15} with $k=2$ (see \cite[equation (3.1)]{buoso-lamberti15}), $\phi=\text{id}$ and $\phi_\epsilon=\phi+\epsilon V$, which is $C^2$, bounded over $\Omega$, as well as its derivatives, and satisfies $\det D\phi_\epsilon\neq0$ for $\epsilon$ small enough.
\end{proof}

\begin{remarque}
\begin{enumerate}
\item Since $M_V$ is real symmetric, it is diagonalisable, hence has $p$ eigenvalues counted with multiplicity. Moreover, the spectrum of $M_V$ does not depend upon the choice of the eigenfunctions $u_k$,...,$u_{k+p-1}$.
\item The crossing of eigenvalue branches (see Figure \ref{fig:vp multiple}) prevents $\Gamma_k^\Omega,...,\Gamma_{k+p-1}^\Omega$ from being differentiable at $0$ in general, even if they are both on the left and on the right. Indeed, their derivative on the left and on the right might not coincide since the bijection with the spectrum of $M_V$ changes when one derives on the left or on the right. In order to overcome this lack of differentiability, it is possible to consider combinations of eigenvalues, called elementary symmetric functions of the eigenvalues. For these, one obtains the stronger Fréchet-differentiability, see \cite[Theorem 3.1]{buoso-lamberti13}.
\item See Theorem 3.5, Lemma 4.1 and formula (4.4) of \cite{ortega-zuazua} for a similar result using slightly different vector fields of deformation than $C^2(\R^d,\R^d)$. Note however that in \cite[Lemma 4.1]{ortega-zuazua}, the eigenfunctions involved in the formula for the derivative seem to depend implicitly on the vector field, which makes the formula less intrinsic than (\ref{eq:vp multiple}).
\end{enumerate}
\end{remarque}

When the eigenvalue under consideration is simple, $M_V$ is a scalar. Consequently, there is plain differentiability, as stated below.

\begin{corollaire}\label{cor:derivee vp simple}
Let $\Omega$ be a $C^4$ bounded open set and $k\in\N^*$. Assume that\/ $\Gamma_k(\Omega)$ is simple and let $u_k$ be an associated $L^2$-normalised eigenfunction. Then, the functional $\Gamma_k^\Omega$ defined in (\ref{eq:Gamma_Omega}) is Gâteau-differentiable at $0$, and its partial derivative in the direction of a vector field $V\in C^2(\R^d,\R^d)$ is
\begin{equation}\label{eq:vp simple}
\partial_V\Gamma_k^\Omega(0)=-\int_{\partial\Omega}\left(\Delta u_{k}\right)^2 V\cdot\vec{n}.
\end{equation}
\end{corollaire}

This result shows that the shape derivative of the first eigenvalue precisely involves the values of the Laplacian of the first eigenfunction (as long as it is unique) on the boundary. But two issues remain. The first is to deal with the volume constraint appearing in (\ref{eq:pb}). To do so, we define, the volume functional $\mathcal{V}^\Omega:W^{1,\infty}(\R^d,\R^d)\to\R$ by
\begin{equation}
\mathcal{V}^\Omega(V)=|(\text{id}+V)\Omega|.
\end{equation}
Then, we build from $\Gamma_k^\Omega$ the functional $G_k^\Omega$ on $C_b^2(\R^d,\R^d)$, the class of $C^2$ vector fields being bounded as well as their derivatives, by setting
\begin{equation}\label{eq:G_Omega}
G_k^\Omega=\left(\mathcal{V}^\Omega\right)^{\frac{4}{d}}\Gamma_k^\Omega.
\end{equation}
It is classical to introduce $G_k^\Omega$ as it essentially behaves as $\Gamma_k^\Omega$ but has the property that $\omega\mapsto G_k^\omega(0)$ is scale-invariant, hence if $\Omega$ is an optimal shape for (\ref{eq:pb}), $0$ minimizes $G_k^\Omega$. Moreover, since the derivative of $\mathcal{V}^\Omega$ is known to be (see \cite[Theorem 5.2.2]{henrot-pierre}), for any $V\in W^{1,\infty}(\R^d,\R^d)$,
\begin{equation}\label{eq:derivee du volume}
\partial_V\mathcal{V}^\Omega(0)=\int_\Omega V\cdot\vec{n},
\end{equation}
we end up with the next corollary.

\begin{corollaire}\label{cor:derivee G_Omega}
With the hypotheses of Corollary \ref{cor:derivee vp simple}, the functional $G_k^\Omega$ defined in (\ref{eq:G_Omega}) is Gâteaux-differentiable at $0$, and its partial derivative in the direction of a vector field $V\in C_b^2(\R^d,\R^d)$ is
\begin{equation}
\partial_VG_k^\Omega(0)=\left[\int_{\partial\Omega}\frac{4\Gamma_k(\Omega)}{d|\Omega|}V\cdot\vec{n}-\int_{\partial\Omega}\left(\Delta u_k\right)^2V\cdot\vec{n}\right]|\Omega|^{\frac{4}{d}}.
\end{equation}
\end{corollaire}

The second issue regarding Corollary \ref{cor:derivee vp simple} is the assumption on the simplicity of $\Gamma_k(\Omega)$. Indeed, as already mentionned, in the context of fourth order elliptic operators, the lack of positivity prevents from using Krein-Rutman Theorem. As a result, one is unable to prove simplicity of the first eigenvalue, which actually fails in general (see \cite[Theorem 3.9]{gazzola-grunau-sweers}). Fortunately, as roughly justified in \cite{mohr}, it can be proved that simplicity holds for the principal eigenvalue on a domain with minimal eigenvalue. The proof of this fact is obtained by contradiction, using the derivative of a multiple eigenvalue. It will be a consequence of the next proposition, which describes a phenomenon of generic \enquote{eigenvalue splitting}, as illustrated in Figure \ref{fig:vp multiple}.

\begin{proposition}\label{prop:derivee vp multiple}
Let $\Omega$ be a \reg \/bounded open set and $k,p\in\N^*$, $p>1$. Assume that $\Gamma_k(\Omega)=...=\Gamma_{k+p-1}(\Omega)=:\Gamma$ and that $\Gamma$ is of multiplicity $p$. Then, there exists $V\in C_b^{2}(\R^d,\R^d)$ such that
\begin{align*}
& \partial_{V}^+\Gamma_k^\Omega(0)<0<\partial_{V}^+\Gamma_{k+p-1}^\Omega(0), \\
& \partial_{V}\mathcal{V}^\Omega(0)=0.
\end{align*}
Here, $\partial_V^+$ denotes the derivative in the direction $V$ on the right.
\end{proposition}

\begin{proof}
We adapt \cite[Lemma 2.5.9]{henrot}. The idea is to use deformations which are localised around two arbitrary boundary points $A_+$ and $A_-$. For that purpose, set $\epsilon>0$, and take a vector field $V_\epsilon=V_\epsilon^++V_\epsilon^-$, where $V_\epsilon^+,V_\epsilon^-\in C_b^2(\R^d,\R^d)$ satisfy:
$$
\supp V_\epsilon^\pm\subseteq B(A_\pm,\epsilon),\qquad \int_{\partial\Omega}V_\epsilon^\pm\cdot\vec{n}=\pm1,\qquad
\int_{\partial\Omega}V_\epsilon\cdot\vec{n}=0.
$$
The last condition immediately gives $\partial_{V_\epsilon}\mathcal{V}^\Omega(0)=0$, in view of (\ref{eq:derivee du volume}). The other conditions tell that $\pm V_\epsilon^\pm\cdot\vec{n}$ is an approximate identity in $A_\pm$ on $\partial\Omega$. More precisely, if $\psi$ is a continuous function over $\partial\Omega$ with $\omega$ as modulus of continuity,
$$
\left|\psi(A_\pm)-\left(\pm\int_{\partial\Omega}\psi V_\epsilon^\pm\cdot\vec{n}\right)\right|\leq\int_{\partial\Omega\cap B(A_\pm,\epsilon)}|\psi-\psi(A_\pm)||V_\epsilon^\pm\cdot\vec{n}|\leq C\omega(\epsilon),
$$
and hence $V_\epsilon^\pm\cdot\vec{n}$ converges to $\pm\delta_{A_\pm}$ in the dual of $C(\partial\Omega)$.
\iffalse
Pour construire les champs de vecteurs $V_\epsilon^\pm$, on se place dans un système de coordonnées donné par une base orthonormée $(\vec{v_1},...\vec{v_d})$ dans lequel $\partial\Omega_0$ est localement dans un voisinage de $A_\pm$ le graphe d'une fonction $\gamma_\pm$. Soit $\chi_\epsilon$ une approximation de l'unité à support dans $B^{\R^{d-1}}(0,\epsilon)$. On pose, dans le nouveau système de coordonnées,
$$
V_\epsilon^\pm(x,\gamma_\pm(x)):=\pm\frac{\chi_\epsilon(x)\vec{v_d}}{\sqrt{1+|\nabla\gamma_\pm|^2}\vec{v_d}\cdot\vec{n}}
$$
De cette façon, on vérifie facilement que les trois conditions demandées sont respectées. De plus, $V_\epsilon^\pm\in C^1(\partial\Omega_0,\R^d)$, et on les complète (voir formule (5.39) de \cite{Henrot:Optimisation_de_forme}) en des champs de vecteurs $C^1(\R^d,\R^d)$ avec les bonnes propriétés de support.
\fi
Now, thanks to elliptic regularity (recall that $\Omega$ is $C^4$) and classical bootstrap arguments we get $u_{k+i-1}\in W^{4,p}(\Omega)$ for any $1<p<\infty$, and hence $u_{k+i-1}\in C^{3}(\overline{\Omega})$ by Sobolev injections. As a consequence, $\Delta u_{k+i-1}$ is continuous over $\partial\Omega$. Then, shrinking $\epsilon\to 0$, we observe that the matrix $M_{V_\epsilon}$ given in (\ref{eq:vp multiple}) converges to the matrix $M$ with coefficients
$$
-\Delta u_{k+i-1}(A_+)\Delta u_{k+j-1}(A_+)+\Delta u_{k+i-1}(A_-)\Delta u_{k+j-1}(A_-),
$$
for $1\leq i,j\leq p$. Setting $X_+$ and $X_-$ to be the vectors with coordinates $\Delta u_{k+i-1}(A_+)$ and $\Delta u_{k+i-1}(A_-)$ respectively, $i=1,...,p$, we obtain the relation
$$
^tXMX=-(X\cdot X_+)^2+(X\cdot X_-)^2,
$$
for any $X\in\R^p$. In particular, the matrix $M$ describes a quadratic form of signature $(1,1)$ as long as $X_+$ and $X_-$ are not colinear. At this point, recall that $X_+$ and $X_-$ depend on $A_+$ and $A_-$, which are arbitrary. We will show that we may tweak $A_+$ and $A_-$ for $X_+$ and $X_-$ not to be colinear. For that, assume by contradiction that $X_+$ and $X_-$ are colinear for any choice of $A_+$ and $A_-$ on the boundary, and let $1\leq i<j\leq p$. Then, the matrix
$$
\begin{pmatrix}
\Delta u_{k+i-1}(A_+) & \Delta u_{k+i-1}(A_-)\\
\Delta u_{k+j-1}(A_+) & \Delta u_{k+j-1}(A_-)
\end{pmatrix}
$$
has determinant $0$. Fixing $A_-$, and setting $c_i:=\Delta u_{k+i-1}(A_-)$ and $c_j:=\Delta u_{k+j-1}(A_-)$, this means that $c_j\Delta u_{k+i-1}(A_+)-c_i\Delta u_{k+j-1}(A_+)=0$ for all $A_+\in\partial\Omega$. In other words, the function $v:=c_ju_{k+i-1}-c_iu_{k+j-1}$ is an eigenfunction of the bilaplacian in $\Omega$ such that $\Delta v=0$ on $\partial\Omega$. Thanks to Corollary \ref{cor:principe de réduction d'ordre}, we obtain that $u_{k+i-1}$ and $u_{k+j-1}$ are colinear, a contradiction.

The previous shows that there exists some points $A_+,A_-\in\partial\Omega$ for which the matrix $M$ has signature $(1,1)$. In other words, the spectrum of $M$ admits a positive and a negative eigenvalue. As a result, for small enough $\epsilon$, $M_{V_\epsilon}$ also admits both positive and negative eigenvalues. Since, by Theorem \ref{thm:derivee vp}, the lowest eigenvalue of $M_{V_\epsilon}$ corresponds to $\partial_{V_\epsilon}^+\Gamma_{k}^\Omega(0)$ and the greatest eigenvalue to $\partial_{V_\epsilon}^+\Gamma_{k+p-1}^\Omega(0)$, we conclude that $
\partial_{V_\epsilon}^+\Gamma_{k}^\Omega(0)<0<\partial_{V_\epsilon}^+\Gamma_{k+p-1}^\Omega(0)$.
\end{proof}

The conclusions of the present section might be combined in order to obtain an information on the function $g_u$ defined in Proposition \ref{prop:principe de réduction d'ordre} in the case of an optimal domain $\Omega$. This is the purpose of the next paragraph.

\section{Simplicity of the eigenvalue and optimality conditions}\label{sec:simplicity}

In this section we prove Theorem \ref{thm:condition d'optimalité}, and discuss the corresponding optimality condition.

\begin{proof}[Proof of Theorem \ref{thm:condition d'optimalité}]
Let $\Omega$ be a $C^4$ optimal shape for problem (\ref{eq:pb}). The simplicity of $\Gamma(\Omega)$ is a direct consequence of Proposition \ref{prop:derivee vp multiple}, as if one had $\Gamma_1(\Omega)=\Gamma_2(\Omega)$, there would exist a vector field $V$ such that $\partial_V\mathcal{V}^\Omega(0)=0$ and $\partial_V^+\Gamma_1^\Omega(0)<0$. Then, we would have $\partial_V^+G_1^\Omega(0)<0$, hence $\Omega$ would not minimise $\omega\mapsto|\omega|^{\frac{4}{d}}\Gamma_1(\omega)$, so it would not solve (\ref{eq:pb}).

Now that simplicity has been proved, we shall invoke Corollary \ref{cor:derivee G_Omega}. Indeed, as $\Omega$ is an optimal shape, $0$ is a critical point of $G_1^\Omega$, hence we get the optimality condition
$$
0=\int_{\partial\Omega}\left[\frac{4\Gamma(\Omega)}{d|\Omega|}-\left(\Delta u\right)^2\right]V\cdot\vec{n},\qquad \forall V\in C_b^2(\R^d,\R^d).
$$
We conclude that $\Delta u=\pm\alpha$ a.e. on $\partial\Omega$. Moreover, since $\Delta u$ is continuous over $\overline{\Omega}$ (recall that, by bootstrap, $u\in W^{4,p}(\Omega)$ for all $1<p<\infty$), it is a.e. constant on each connected component of $\partial\Omega$.
\end{proof}

Note that the optimality condition given in Theorem \ref{thm:condition d'optimalité} is actually fulfilled by any $C^4$ regular shape $\Omega$ with simple principal eigenvalue and such that $0$ is a critical point for $G_1^\Omega$. This motivates the following definition.

\begin{definition}
An open set $\Omega$ is a critical shape (for the principal eigenvalue) if any $L^2$-normalised first eigenfunction $u$ on $\Omega$ is such that $\Delta u$ is a.e. constant equal to $\pm\sqrt{\frac{4\Gamma(\Omega)}{d|\Omega|}}$ on each connected component of $\partial\Omega$.
\end{definition}

\begin{remarque}
Any ball $B$ is a critical shape since, according to Propostion \ref{prop:fonction propre boule}, the first eigenfunction is radial (derive $G_1^B$ in the direction of a radially symmetric vector field). See also \cite{buoso} for more general results.
\end{remarque}

Considering the order reduction principle proved in Section \ref{sec:order reduction} and the optimality condition derived in the present section, we are equipped for proving Theorem \ref{thm:faber-krahn moyenne} and Theorem \ref{thm:surdetermine}.

\section{Proofs of Theorem \ref{thm:faber-krahn moyenne} and Theorem \ref{thm:surdetermine}}\label{sec:theorem}

In this section, we combine the order reduction principle (Proposition \ref{prop:principe de réduction d'ordre}) and the optimality condition (Theorem \ref{thm:condition d'optimalité}) to provide proofs for Theorem \ref{thm:faber-krahn moyenne} and Theorem \ref{thm:surdetermine}. Let us begin with the most straightforward, which is undoubtedly Theorem \ref{thm:surdetermine}. With Theorem \ref{thm:condition d'optimalité} in mind, we see that it is enough to prove Theorem \ref{thm:surdetermine} for critical shapes, which is performed below.

\begin{theoreme}\label{thm:surdetermine critique}
Let $\Omega$ be a critical shape satisfying (\ref{hyp:rg}). Let $u$ be a first eigenfunction on $\Omega$ such that $\partial_n\Delta u$ is constant on $\partial\Omega$. Then, $\Omega$ is a ball.
\end{theoreme}

\begin{proof}
Without loss of generality, we assume $u$ to be $L^2$-normalised. Since $\Omega$ is a critical shape, we know that $\Delta u$ is a.e. constant on each connected component of $\partial\Omega$. But $\partial\Omega$ is assumed connected, hence $\Delta u$ is constant on $\partial\Omega$ equal to $\pm\alpha$, where $\alpha=\sqrt{\frac{4\Gamma(\Omega)}{d|\Omega|}}$. Considering $-u$ if needed, we shall assume that $\Delta u=\alpha$ a.e. on $\partial\Omega$, and, consequently, $g_u=\sqrt{\frac{4}{d|\Omega|}}>0$ a.e. not only on $\partial\Omega$ but in the whole $\Omega$. Applying the order reduction principle (Proposition \ref{prop:principe de réduction d'ordre}), we obtain that $z_u=\frac{\Delta}{\sqrt{\mu}}u+u-g_u$ is negative and satisfies (\ref{eq:pb simplifie}). Moreover, the fact that $\partial_n\Delta u$ remains constant on the boundary (combined with the fact that $g_u$ is constant) shows that $\partial_n z_u$ is constant on $\partial\Omega$. Thus $z_u<0$ satisfies an overdetermined problem of order 2, and we conclude applying Serrin's Theorem \cite[Theorem 2]{serrin}.
\end{proof}

We now turn to the proof of Theorem \ref{thm:faber-krahn moyenne}. To do so, we use the variational formulation of the first eigenvalue involving $z_u$ given by Proposition \ref{prop:principe de réduction d'ordre}. This new expression is interesting in the sense that it allows using symmetrisation techniques available for one-sign $H_0^1(\Omega)$ functions. That's why we recall the Schwarz symmetrisation (see the classical \cite{kawohl} for a general discussion on level set rearrangements).

\begin{definition}\label{def:schwarz}
Let $\Omega$ be an open set and $u$ be a measurable function on $\Omega$. Let $B$ be a ball of same volume than $\Omega$. The nonincreasing spherical symmetric rearrangment (also called Schwarz symmetrisation) of $u$ is the measurable function $u^*$ defined on $B$ such that its radial part is the generalised inverse of the distribution function $\mu_u$ of $u$, that is
$$
u^*(x):=\mu_u^{[-1]}(|B_{|x|}|)=\inf\{t:\mu_u(t)\leq |B_{|x|}|\}=\inf\{t:|\{u>t\}|\leq |B_{|x|}|\},
$$
where $B_r$ denotes the ball of radius $r$ and of same center as $B$. We recall that $u$ and $u^*$ are equimeasurable and that if $u\in H_0^1(\Omega)$ is nonnegative, $u^*\in H_0^1(B)$. Moreover, for any $z\in H_0^1(\Omega)$, we define $z^\#:=-(-z)^*$.
\end{definition}

Then, Theorem \ref{thm:faber-krahn moyenne} will be a consequence of the following result.

\begin{theoreme}\label{thm:faber-krahn moyenne critique}
Let $\Omega$ be a critical shape satisfying (\ref{hyp:rg}) and $B$ a ball such that $|\Omega|=|B|$. Let $u$ be a first $L^2$-normalised eigenfunction on $\Omega$ and $u_B$ a first $L^2$-normalised eigenfunction on $B$. Assume that
\begin{equation}\label{eq:inverse M}
\left|\int_\Omega u\right|\leq \left|\int_B u_B\right|.
\end{equation}
Then, inequality (\ref{eq:conjecture}) holds. Moreover, if (\ref{eq:inverse M}) is strict, (\ref{eq:conjecture}) is also strict. Finally, if\/ \mbox{$\Gamma(\Omega)=\Gamma(B)$}, $\Omega$ has to be a translation of $B$.
\end{theoreme}

\begin{proof}
Proceeding as in the beginning of the proof of Theorem \ref{thm:surdetermine critique}, we might assume that $g_u=\sqrt{\frac{4}{d|\Omega|}}>0$ a.e. in $\Omega$. Since $B$ is also a critical shape satisfying (\ref{hyp:rg}) (recall the remark below the definition of critical shapes), the same applies to $u_B$, and we conclude that we shall also take $g_{u_B}=\sqrt{\frac{4}{d|\Omega|}}$ a.e. in $B$. As a result of $g_{u_B}\geq0$, observe that $u_B$ (which is one-sign) is positive due to the maximum principle and to Hopf Boundary Lemma.

On the other hand, since $g_u\geq0$, we get $z_u<0$, and $z_u^\#$ is a negative $H_0^1(B)$ function. Moreover, the properties of the Schwarz symmetrisation ensure that $\int_\Omega|\nabla z_u|^2\geq\int_{B}|\nabla z_u^\#|^2$, $\int_\Omega z_u^2=\int_{B} (z_u^\#)^2$, and $\int_\Omega z_u=\int_{B} z_u^\#$. Therefore, thanks to Proposition \ref{prop:principe de réduction d'ordre},
$$
\frac{1}{\sqrt{\Gamma(\Omega)}}=-\frac{\int_{\Omega} |z_u|^2+g_u\int_{\Omega} z_u}{\int_\Omega|\nabla z_u|^2}\leq-\frac{\int_{B} |z_u^\#|^2+g_u\int_{B} z_u^\#}{\int_B|\nabla z_u^\#|^2}\leq-\min_{z\in H_0^1(B)}\frac{\int_{B} z^2+g_{u_B}\int_{B} (2z-z_u^\#)}{\int_B|\nabla z|^2}.
$$
Note that the numerator in the above quotients is always nonpositive (from the first equality), which justifies the first inequality. Now, we claim that $\int_B z_u^\#\leq\int_B z_{u_B}$. Indeed, if true, this result would lead to
$$
\frac{1}{\sqrt{\Gamma(\Omega)}}\leq-\min_{z\in H_0^1(B)}\frac{\int_{B} z^2+g_{u_B}\int_{B} (2z-z_{u_B})}{\int_B|\nabla z|^2}=\frac{1}{\sqrt{\Gamma(B)}},
$$
the last equality coming once again from Proposition \ref{prop:principe de réduction d'ordre} applied to $B$. This would in turn give the Faber-Krahn inequality $\Gamma(\Omega)\geq\Gamma(B)$. Note also that if $\int_B z_u^\#\leq\int_B z_{u_B}$ is strict, then $\Gamma(\Omega)\geq\Gamma(B)$ is also strict.

Hence it remains only to prove that $\int_B z_u^\#\leq\int_B z_{u_B}$. But thanks to the properties of the Schwarz rearrangement, $\int_B z_u^\#=\int_\Omega z_u$. Then, using the expression of $z_u$ combined with the fact that $\int_\Omega\Delta u=0$, we find $\int_B z_u^\#=\int_\Omega u-|\Omega|g_u$. In the same way, $\int_Bz_{u_B}=\int_B u_B-|B|g_{u_B}$. Thus, as $|\Omega|=|B|$ and $g_u=g_{u_B}$, we obtain that $\int_B z_u^\#\leq\int_B z_{u_B}$ if and only if $\int_\Omega u\leq\int_B u_B$, which holds by assumption (recall that $u_B>0$). Moreover, if one of these inequalities is strict, the other also holds strictly.

Lastly, if $\Gamma(\Omega)=\Gamma(B)$, all our inequalities become equalities. In particular, $\int_\Omega|\nabla z_u|^2=\int_B|\nabla z_u^\#|^2$, thus we apply \cite[Theorem 2.2]{friedman-mcleod}. This is possible since, on the one hand, as $u$ is analytic in $\Omega$, $z_u$ is also analytic, hence $|\{z_u=t\}|=0$ for all $\inf z_u<t<\sup z_u$. On the other hand, thanks to elliptic regularity \cite[Theorem 2.20]{gazzola-grunau-sweers} and to classical bootstrap arguments, $u$ is actually $W^{4,p}(\Omega)$, and in particular $C^{3,\gamma}(\overline{\Omega}),0<\gamma<1$ due to Sobolev embeddings. Finally, $z_u$ is Lipschitz in $\R^d$. Then, \cite[Theorem 2.2]{friedman-mcleod} yields that, up to translation, $z_u=z_u^\#$, and in particular that $\Omega$ is a ball.
\end{proof}

\begin{proof}[Proof of Theorem \ref{thm:faber-krahn moyenne}]
If $\Omega$ is an optimal shape satisfying (\ref{hyp:rg}), Theorem \ref{thm:condition d'optimalité} shows that $\Gamma(\Omega)$ is simple and that $\Omega$ is a critical shape. Assume that its $L^2$-normalized principal eigenfunction $u$ does not verify (\ref{hyp:moyennes}). Theorem \ref{thm:faber-krahn moyenne critique} then applies and shows that $\Gamma(\Omega)> \Gamma(B)$, contradicting the optimality of $\Omega$. Therefore, (\ref{hyp:moyennes}) holds. Moreover, in case of equality, Theorem \ref{thm:faber-krahn moyenne critique} still applies and gives $\Gamma(\Omega)\geq \Gamma(B)$. As $\Omega$ is optimal, we conclude that $\Gamma(\Omega)=\Gamma(B)$, hence Theorem \ref{thm:faber-krahn moyenne critique} implies that $\Omega=B$ up to a translation.
\end{proof}

Theorem \ref{thm:faber-krahn moyenne} relies on the central assumption that equality is attained in (\ref{hyp:moyennes}), or in other words that the converse of (\ref{hyp:moyennes}) holds. Unfortunately, the inequality $\left|\int_\Omega u\right|\leq\left|\int_B u_B\right|$ seems not easy to check in general. For instance, to estimate the mean value of $u$ on the optimal domain $\Omega$, one could try to use the inequality
\begin{equation}\label{eq:estimation de la moyenne}
\int_\Omega u\leq g_u|\Omega|=\sqrt{\frac{4|\Omega|}{d}},
\end{equation}
coming from the fact that $\int_\Omega z_u\leq 0$ (recall Proposition \ref{prop:principe de réduction d'ordre}). However, as $B$ is a critical shape, $u_B$ satisfies (\ref{eq:estimation de la moyenne}) as well, hence it is illusory to intend showing the reverse $\sqrt{\frac{4|\Omega|}{d}}\leq\int_B u_B$, since it would mean that $z_{u_B}=0$, in contradiction with Proposition \ref{prop:principe de réduction d'ordre}.

Nevertheless, even if having equality in (\ref{hyp:moyennes}) is a restrictive condition, Theorem \ref{thm:faber-krahn moyenne} has two interesting consequences that we shall explain in section \ref{sec:corollaires}.

\section{Consequences of Theorem \ref{thm:faber-krahn moyenne}}\label{sec:corollaires}

The first immediate corollary of Theorem \ref{thm:faber-krahn moyenne} regards the volume of one of the nodal domains of $u$.

\begin{corollaire}\label{corollaire:faber-krahn volume}
With the hypotheses of Theorem \ref{thm:faber-krahn moyenne}, if $\int_\Omega u>0$, writing $\Omega_+:=\{u>0\}$, then
\begin{equation}\label{eq:Omega+ max}
\sqrt{|\Omega_+|}>\int_B u_B.
\end{equation}
\end{corollaire}

\begin{proof}
Assume by contradiction that $\sqrt{|\Omega_+|}\leq\int_B u_B$. Then, $\int_\Omega u\leq\int_{\Omega_+} u\leq\sqrt{|\Omega_+|}\sqrt{\int_{\Omega_+}u^2}\leq\sqrt{|\Omega_+|}\leq\int_B u_B\leq\int_\Omega u$, the last inequality coming from Theorem \ref{thm:faber-krahn moyenne}. Therefore, all the previous inequalities, in particular Hölder's, are equalities. This means that $u=1$ in $\Omega_+=B$, a contradiction. 
\end{proof}

This result confirms that it might be interesting to evaluate the mean value of $u_B$. This is possible since $u_B$ shall be computed explicitly as it is stated in the next result, the proof of which is detailed in appendix page \pageref{preuve:fonction propre boule}.

\begin{restatable}{proposition}{fonctionPropreBoule}\label{prop:fonction propre boule}
Let $B$ be the ball $B(0,R)$. The first $L^2$-normalised eigenfunction $u_B$ is radially symmetric. Moreover, $\Gamma(B)$ and (up to sign) $u_B$ are given by the formulas
\begin{equation}\label{eq:uB}
\Gamma(B)=\frac{\gamma_\nu^4}{R^4},\qquad u_B(r)=\frac{1}{\sqrt{d|B|}}\left[\frac{J_\nu(k_\nu r)}{J_\nu(k_\nu R)}-\frac{I_\nu(k_\nu r)}{I_\nu(k_\nu R)}\right]\left(\frac{r}{R}\right)^{-\nu},
\end{equation}
where $\nu:=d/2-1$, $J_\nu$ and $I_\nu$ stand for the Bessel and modified Bessel functions of order $\nu$, and $k_\nu :=\gamma_\nu/R$, $\gamma_\nu$ being the first positive zero of $f_\nu$ defined by
$$
f_\nu(r)=\left[\frac{J_{\nu+1}}{J_\nu}(r)+\frac{I_{\nu+1}}{I_\nu}(r)\right]r^{d-1}.
$$
Finally,
\begin{equation}\label{eq:int uB}
\int_B u_B=\frac{\sqrt{d|B|}}{\gamma_\nu}\left[\frac{J_{\nu+1}}{J_\nu}(\gamma_\nu)-\frac{I_{\nu+1}}{I_\nu}(\gamma_\nu)\right]=2\frac{\sqrt{d|B|}}{\gamma_\nu}\frac{J_{\nu+1}}{J_\nu}(\gamma_\nu).
\end{equation}
\end{restatable}

Note that it is easy to evaluate numerically (\ref{eq:int uB}). Indeed, in Python 3 for instance, the package \verb|special|, from the module \verb|scipy|, directly provides the functions \verb|jv| and \verb|iv|, corresponding respectively to the Bessel functions $J_\nu$ and $I_\nu$. Then it remains to compute $\gamma_\nu$, but this can be done by dichotomy thanks to (\ref{eq:entrelacement des zeros}) as long as one knows $j_{\nu,1}$ and $j_{\nu,2}$ (where, for $n\in\N^*$, $j_{\nu,n}$ are the positive zeros of $J_\nu$). For that observe, as explained in Theorem 2.1 and Theorem 2.2 of \cite{ikebe-kikuchi-fujishiro}, that the zeros of $J_\nu$  can be approximated by computing the eigenvalues of some matrix.

In Table \ref{tab:int uB} is given the value of $\int_B u_B$ in the case where $B$ is the ball of volume $1$. We also give the minimum volume allowed for $\Omega_+$ to satisfy (\ref{eq:Omega+ max}), that is $(\int_B u_B)^2$.

\begin{table}[h!]
$$
\begin{array}{c|c|c}
d & \int_B u_B & (\int_B u_B)^2 \\
\hline
4 & 0.6056 & 0.3668 \\
5 & 0.5643 & 0.3185 \\
6 & 0.5308 & 0.2817 \\
7 & 0.5028 & 0.2528 \\
8 & 0.4790 & 0.2294 \\
9 & 0.4583 & 0.2101
\end{array}
$$
\caption{Value of $\int_B u_B$ for several dimensions. Here, $B$ is chosen to be the ball of volume $1$.}
\label{tab:int uB}
\end{table}

Let us now discuss another consequence of Theorem \ref{thm:faber-krahn moyenne}. In view of proving the Rayleigh Conjecture, Theorem \ref{thm:faber-krahn moyenne} tells that the last step would be to show the converse of (\ref{hyp:moyennes}) for an optimal shape $\Omega$. On the other hand, the optimality of $\Omega$ means, by definition, that
$$
\int_{\Omega}|\Delta u|^2\leq\int_B|\Delta u_B|^2,
$$
where $u$ (resp. $u_B$) is an $L^2$-normalised first eigenfunction on $\Omega$ (resp. $B$). From this inequality, one shall wonder whether it is possible to deduce the converse of (\ref{hyp:moyennes}). This problem is actually not so far from a maximum principle type property, which classically asserts that if $v_1,v_2\in H_0^1(\omega)$ satisfy $-\Delta v_1\leq-\Delta v_2$, then $v_1\leq v_2$ in $\omega$. In our situation, it would be desirable to convert these pointwise inequalities into integral ones. Therefore, even if it does not immediately answer our initial concern, it would be interesting to study to which extent the following $L^p$ norm and mean value formulations of the maximum principle hold: for $v_1\in W_0^{1,p}(\omega_1)\cap W^{2,p}(\omega_1)$ and $v_2\in W_0^{1,p}(\omega_2)\cap W^{2,p}(\omega_2)$,
\begin{align}
\int_{\omega_1}|\Delta v_1|^p\leq\int_{\omega_2}|\Delta v_2|^p\qquad & \Longrightarrow \qquad \int_{\omega_1} |v_1|^p\leq\int_{\omega_2} |v_2|^p,\label{eq:principe du max Lp norme}\\
\int_{\omega_1}(-\Delta v_1)^p\leq\int_{\omega_2}(-\Delta v_2)^p\qquad & \Longrightarrow \qquad \int_{\omega_1} v_1^p\leq\int_{\omega_2} v_2^p.
\label{eq:principe du max Lp moyen}
\end{align}
At this time, we were not able to answer the above (quite vague) questions, and could only argue that (\ref{eq:principe du max Lp moyen}) cannot hold in full generality for $p=1$, since it would imply that any $H_0^2$ function has zero mean value. Anyway, in the remaining, we will state an interesting consequence of Theorem \ref{thm:faber-krahn moyenne} using the standard maximum principle combined with Talenti's comparison principle, which we recall below (see \cite{talenti76}).

\begin{theoreme}\label{thm:talenti}
Let $\omega$ be an open set and $\omega^*$ its Schwarz symmetrisation. Let $f\in L^2(\omega)$ and $u\in H^2(\omega)$ the solution of
\begin{equation*}
\begin{cases}
-\Delta u = f & in\quad\omega,\\
u=0 & on\quad\partial\omega.
\end{cases}
\end{equation*}
Let $f^*,u^*\in L^2(\omega^*)$ be the Schwarz symmetrisations of $f,u$ and let $v\in H^2(\omega^*)$ solve
\begin{equation*}
\begin{cases}
-\Delta v = f^* & in\quad \omega^*,\\
v=0 & on\quad\partial \omega^*.
\end{cases}
\end{equation*}
Assume that $u\geq0$. Then,
\begin{equation*}
v\geq u^*\qquad\text{a.e. in }\omega^*.
\end{equation*}
\end{theoreme}

\begin{remarque}
\begin{enumerate}
\item The hypothesis $u\geq0$ is not precised in \cite{talenti76}, but it is mentionned in \cite[Theorem 3.1.1]{kesavan}. This comes from the definition of Schwarz symmetrisation for signed functions, which differs in both references. Here, in view of Definition \ref{def:schwarz}, we conform to the convention adopted in \cite{kesavan}.

\item We mention that, as long as $u$ is assumed nonnegative, Schwarz symmetrisation might be replaced by Talenti symmetrisation (see \cite{talenti81}), which is defined in the following way: let $f\in L^2(\omega)$, we set, for all $s\in[0,|\omega|[$,
$$
f^\dagger(s)=f_+^*(s)-f_-^*(|\omega|-s).
$$
Then, the Talenti symmetrisation of $f$ is the function $f^\dagger$ defined on $\omega^*$ by $\forall x\in \omega^*$,
$$
f^\dagger(x):=f^\dagger(|B_{|x|}|).
$$
\end{enumerate}
\end{remarque}

\begin{corollaire}
With the hypotheses of Theorem \ref{thm:faber-krahn moyenne}, assume without loss of generality that $u_B>0$ and that $\int_\Omega u>0$. Writing $\Omega_+:=\{u>0\}$ and $\Omega_+^*$ its Schwarz symmetrisation, if $(-\Delta u|_{\Omega_+})^*\leq -\Delta u_B$ in $\Omega_+^*$, then, up to a translation, $\Omega=B$.
\end{corollaire}

\begin{remarque}
\begin{enumerate}
\item
We stress that if $\Omega$ is an optimal shape, then
$$
\int_{B}\left|(-\Delta u)^*\right|^2=\int_{\Omega}(\Delta u)^2=\Gamma(\Omega)\leq\Gamma(B)=\int_B(-\Delta u_B)^2.
$$
The assumption of the corollary is then a pointwise version of this inequality.

\item As we shall see in the proof, the assumption $(-\Delta u|_{\Omega_+})^*\leq -\Delta u_B$ is used for applying the maximum principle, which in turn yields a pointwise inequality between $u$ and $u_B$, although only an inequality in terms of mean value is actually needed to invoke Theorem \ref{thm:faber-krahn moyenne}. That's why, if one could prove some \enquote{mean value maximum principle} as in (\ref{eq:principe du max Lp norme}) and (\ref{eq:principe du max Lp moyen}), one could hope to drop the assumption.

\item As in Theorem \ref{thm:talenti}, Schwarz symmetrisation $^*$ might be replaced by Talenti's one $^\dagger$.
\end{enumerate}
\end{remarque}

\begin{proof}
We set $f:=-\Delta u|_{\Omega_+}$. Let $v$ be the $H_0^1(\Omega_+^*)$ solution of the problem $-\Delta v=f^*$ in $\Omega_+^*$. According to Talenti's comparison principle, since $u\geq 0$ in $\Omega_+$, $v\geq u^*$ in $\Omega_+^*$.

Then, as $-\Delta v=(-\Delta u|_{\Omega_+})^*\leq-\Delta u_B$ in $\Omega_+^*$, we get $-\Delta (u_B-v)\geq 0$ in $\Omega_+^*$. Thus the maximum principle forces $u_B-v$ to reach its minimum value on the boundary of $\Omega_+^*$. Moreover, $u_B\geq 0$ on $B$ and hence on $\partial\Omega_+^*$. Therefore $u_B-v\geq 0$ on $\partial\Omega_+^*$. To conclude, $u_B\geq v$ not only on the boundary, but in the whole $\Omega_+^*$.

We obtained that $u^*\leq u_B$ pointwisely in $\Omega_+^*$. In particular, since $u_B\geq0$,
$$
\int_\Omega u\leq\int_{\Omega_+} u=\int_{\Omega_+^*}u^*\leq\int_{\Omega_+^*} u_B\leq\int_B u_B,
$$
and we conclude thanks to Theorem \ref{thm:faber-krahn moyenne}.
\end{proof}

\appendix
\setcounter{secnumdepth}{0}
\section{Appendix}

\begin{proof}[Proof of Proposition \ref{prop:fonction propre boule}]\label{preuve:fonction propre boule}
For readability, we ommit the subscript $B$ in $u_B$. According to \cite{ashbaugh-benguria}, in $B$, the first eigenfunction is radially symmetric and of the form $\forall r\in[0,R[$,
$$
u(r)=\left(aJ_\nu(k r)+bI_\nu(k r)\right)r^{-\nu},
$$
where $k:=\Gamma(B)^{\frac{1}{4}}$. Then, using the identities $J_\nu'(x)=\frac{\nu J_\nu(x)}{x}-J_{\nu+1}(x)$ and $I_\nu'(x)=\frac{\nu I_\nu(x)}{x}+I_{\nu+1}(x)$ we find
$$
\partial_r u(r)=\left(-aJ_{\nu+1}(kr)+bI_{\nu+1}(kr)\right)kr^{-\nu}.
$$
Now, for $u$ to fulfill the condition $u(R)=\partial_r u(R)=0$ although being non trivial, one observes that the matrix
$$
M=\left(\begin{array}{cc}
J_\nu(kR) & I_\nu(kR) \\
-J_{\nu+1}(kR) & I_{\nu+1}(kR)
\end{array}\right)
$$
needs having a non trivial kernel. In other words, its determinant needs to vanish, hence
$$
f_\nu(kR)=J_\nu(kR)I_{\nu+1}(kR)+J_{\nu+1}(kR)I_\nu(kR)=0.
$$
Conversely, as soon as $k$ satisfies this equation, $u$ will be solution of an eigenvalue problem in $B$ with Dirichlet boundary conditions. Consequently, $k$ is necessarily the lowest positive solution of this equation, meaning that $k=k_\nu$. Hence $\Gamma(B)=k_\nu^4$.

We also invoke the article \cite[equation (2.2)]{baricz-ponnusamy-singh} according to which the positive zeros $\gamma_{\nu,n}$ of $f_{\nu}$ and the positive zeros $j_{\nu,n}$ of $J_\nu$ interlace in the following way
\begin{equation}\label{eq:entrelacement des zeros}
j_{\nu,n}<\gamma_{\nu,n}<j_{\nu,n+1}.
\end{equation}
In particular, $M\neq0$ and it has a one-dimensional kernel generated, in virtue of the identity $u(R)=0$, by the vector $(I_\nu(\gamma_\nu),-J_\nu(\gamma_\nu))$ or equivalently by the vector $R^\nu(J_\nu(\gamma_\nu)^{-1},-I_\nu(\gamma_\nu)^{-1})$. In other words, there exists a real number $\beta$ such that
$$
\left(\begin{array}{c}
a \\
b 
\end{array}\right)
=\beta R^\nu
\left(\begin{array}{c}
J_\nu(\gamma_\nu)^{-1}  \\
-I_\nu(\gamma_\nu)^{-1} 
\end{array}\right)$$
Finding the values of $a$ and $b$ is thus equivalent to determining $\beta$. For that purpose, we use the normalisation of $u$, i.e.
\begin{align}\label{eq:normalisation}
1=\int_B u^2=\beta^2R^{2\nu}|\mathbb{S}^{d-1}|&\left[J_\nu(\gamma_\nu)^{-2}\int_0^RJ_\nu(k r)^2r^{d-2\nu-1}\right.\nonumber\\
&\left.+I_\nu(\gamma_\nu)^{-2}\int_0^RI_\nu(k r)^2r^{d-2\nu-1}\right.\\
&\left.-2J_\nu(\gamma_\nu)^{-1}I_\nu(\gamma_\nu)^{-1}\int_0^RI_\nu(kr)J_\nu(kr)r^{d-2\nu-1}\right].\nonumber
\end{align}
As $d-2\nu-1=1$, it turns out that we need to compute the integral of product of Bessel functions against $r$. That's why we use the Gradshteyn and Ryzhik collection \cite[section 6.521, formula 1]{gradshteyn-ryzhik}, that is, for all $\alpha\neq\beta\in\mathbb{C}$ and $\nu>-1$,
\begin{equation}\label{eq:Gradshteyn 1}
\int_0^1xJ_\nu(\alpha x)J_\nu(\beta x)=\frac{\beta J_{\nu-1}(\beta)J_\nu(\alpha)-\alpha J_{\nu-1}(\alpha)J_\nu(\beta)}{\alpha^2-\beta^2}=\frac{\alpha J_{\nu+1}(\alpha)J_\nu(\beta)-\beta J_{\nu+1}(\beta)J_\nu(\alpha)}{\alpha^2-\beta^2}.
\end{equation}
We apply this formula with $\alpha=i\gamma_\nu$ and $\beta=\gamma_\nu$, and find
$$
\int_0^RI_\nu(k_\nu r)J_\nu(k_\nu r)r^{d-2\nu-1}=\frac{R^2}{2\gamma_\nu}[I_{\nu+1}(\gamma_\nu)J_\nu(\gamma_\nu)+J_{\nu+1}(\gamma_\nu)I_\nu(\gamma_\nu)]=\frac{R^2}{2\gamma_\nu}f_\nu(\gamma_\nu)=0.
$$
For the other integrals, we first remark that when ($\alpha,\beta\in\R$ and) $\alpha\to\beta$ in (\ref{eq:Gradshteyn 1}), one obtains
\begin{equation}\label{eq:Gradshteyn alpha2beta}
\int_0^1 xJ_\nu(\beta x)^2=\frac{J_\nu(\beta)^2}{2\beta}\frac{d}{d\beta}\left[\frac{\beta J_{\nu+1}(\beta)}{J_\nu(\beta)}\right]=\frac{1}{2}\left[J_{\nu+1}(\beta)^2+J_\nu(\beta)^2-\frac{\nu}{\beta} J_{\nu+1}(\beta)J_\nu(\beta)\right].
\end{equation}
Hence, with $\beta=\gamma_\nu$, we find
$$
\int_0^RJ_\nu(k_\nu r)^2r^{d-2\nu-1}=\frac{R^2}{2}\left[J_{\nu+1}(\gamma_\nu)^2+J_\nu(\gamma_\nu)^2-\frac{\nu}{\gamma_\nu} J_{\nu+1}(\gamma_\nu)J_\nu(\gamma_\nu)\right].
$$
But because both extremal members in (\ref{eq:Gradshteyn alpha2beta}) depend holomorphicly on $\beta$, this formula remains true even when $\beta\in\C$ thanks to the isolation of zeros, hence, we can apply it to $\beta=i\gamma_\nu$:
$$
\int_0^RI_\nu(k_\nu r)^2r^{d-2\nu-1}=\frac{R^2}{2}\left[-I_{\nu+1}(\gamma_\nu)^2+I_\nu(\gamma_\nu)^2-\frac{\nu}{\gamma_\nu} I_{\nu+1}(\gamma_\nu)I_\nu(\gamma_\nu)\right].
$$
Finally, since $f_\nu(\gamma_\nu)=0$, the term between the brackets in (\ref{eq:normalisation}) becomes
$$
\begin{array}{rcll}
& \frac{R^2}{2} & \left[J_\nu(\gamma_\nu)^{-2}\left(J_{\nu+1}(\gamma_\nu)^2+J_\nu(\gamma_\nu)^2-\frac{\nu}{\gamma_\nu}J_{\nu+1}(\gamma_\nu)J_\nu(\gamma_\nu)\right) \right.\\
& & \left. +I_\nu(\gamma_\nu)^{-2}\left(-I_{\nu+1}(\gamma_\nu)^2+I_\nu(\gamma_\nu)^2-\frac{\nu}{\gamma_\nu}I_{\nu+1}(\gamma_\nu)I_\nu(\gamma_\nu)\right)\right] \\
= & \frac{R^2}{2} & \left[2+\left(\frac{J_{\nu+1}}{J_\nu}(\gamma_\nu)-\frac{I_{\nu+1}}{I_\nu}(\gamma_\nu)-\frac{\nu}{\gamma_\nu}\right)\left(\frac{J_{\nu+1}}{J_\nu}(\gamma_\nu)+\frac{I_{\nu+1}}{I_\nu}(\gamma_\nu)\right)\right] \\
= & R^2 & .
\end{array}
$$
Using $|\mathbb{S}^{d-1}|R^d=d|B|$, we have that $\beta^{-2}=|\mathbb{S}^{d-1}|R^{2\nu+2}=d|B|$, hence
\begin{equation}\label{eq:A et B}
a=\frac{R^{\nu}}{J_\nu(\gamma_\nu)\sqrt{d|B|}},\qquad b=-\frac{R^{\nu}}{I_\nu(\gamma_\nu)\sqrt{d|B|}}.
\end{equation}
In particular,
$$
u(r)=\frac{1}{\sqrt{d|B|}}\left(\frac{J_\nu(k_\nu r)}{J_\nu(k_\nu R)}-\frac{I_\nu(k_\nu r)}{I_\nu(k_\nu R)}\right)\left(\frac{r}{R}\right)^{-\nu}.
$$
which corresponds to (\ref{eq:uB}). After having obtained the expression of $u$, we would like to compute its integral. Observing that
$$
\int_B u=\frac{1}{\Gamma(B)}\int_B\Delta^2 u=\frac{1}{\Gamma(B)}\int_{\partial B}\partial_n\Delta u=\frac{R^{d-1}|\mathbb{S}^{d-1}|}{k_\nu^4}\partial_r\Delta u(R),
$$
it remains only to compute
$$
\partial_r\Delta u(r)=[aJ_{\nu+1}(k_\nu r)+bI_{\nu+1}(k_\nu r)]k_\nu^3r^{-\nu},
$$
for which we used the identities $J_{\nu+1}'(x)=J_\nu(x)-\frac{\nu+1}{x}J_{\nu+1}(x)$ and $I_{\nu+1}'(x)=I_\nu(x)-\frac{\nu+1}{x}I_{\nu+1}(x)$. As a result, we get
\begin{align*}
\int_B u = & \frac{R^{d-1}|\mathbb{S}^{d-1}|}{k_\nu^4}\frac{k_\nu^3}{\sqrt{d|B|}}\left[\frac{J_{\nu+1}}{J_\nu}(\gamma_\nu)-\frac{I_{\nu+1}}{I_\nu}(\gamma_\nu)\right] \\
= & \frac{\sqrt{d|B|}}{\gamma_\nu}\left[\frac{J_{\nu+1}}{J_\nu}(\gamma_\nu)-\frac{I_{\nu+1}}{I_\nu}(\gamma_\nu)\right].
\end{align*}
Note that the last equality in (\ref{eq:int uB}) comes from the fact that $f_\nu(\gamma_\nu)=0$.
\iffalse
$$
\int_B u= a|\mathbb{S}^{d-1}|\int_0^rJ_\nu(k r)r^{d-1-\nu}+b|\mathbb{S}^{d-1}|\int_0^rI_\nu(k r)r^{d-1-\nu}.
$$
Observing that $d-1-\nu=\nu+1$, we will use advantageously formulas 5 and 7 from section 6.561 of \cite{gradshteyn-ryzhik}, which read,
$$
\int_0^1x^{\nu+1}J_\nu(\alpha x)=\alpha^{-1}J_{\nu+1}(\alpha),\qquad \int_0^1x^{\nu+1}I_\nu(\alpha x)=\alpha^{-1}I_{\nu+1}(\alpha).
$$
Applying these formulas to $\alpha=\gamma_\nu$, we get
$$
\int_0^RJ_\nu(k r)r^{d-1-\nu}=\frac{R^{d-\nu}}{\gamma_\nu}J_{\nu+1}(\gamma_\nu),\qquad \int_0^RI_\nu(k r)r^{d-1-\nu}=\frac{R^{d-\nu}}{\gamma_\nu}I_{\nu+1}(\gamma_\nu).
$$
This, combined with (\ref{eq:A et B}) gives the desired formula
\begin{align*}
\int_B u= & \frac{R^{\nu}}{\sqrt{d|B|}}|\mathbb{S}^{d-1}|\frac{R^{d-\nu}}{\gamma_\nu}\left[\frac{J_{\nu+1}}{J_\nu}(\gamma_\nu)-\frac{I_{\nu+1}}{I_\nu}(\gamma_\nu)\right] \\
= & \frac{\sqrt{d|B|}}{\gamma_\nu}\left[\frac{J_{\nu+1}}{J_\nu}(\gamma_\nu)-\frac{I_{\nu+1}}{I_\nu}(\gamma_\nu)\right].
\end{align*}
Note that the last equality in (\ref{eq:int uB}) comes from the fact that $f_\nu(\gamma_\nu)=0$.
\fi
\end{proof}

%\input{parties/annexe}

%\input{parties/conclusion}

%% <== End of main content
%%%%%%%%%%%%%%%%%%%%%%%%%%%%%%%%%%%%%%%%%%%%%%%%%%%%%%%%%%%%%

%%%%%%%%%%%%%%%%%%%%%%%%%%%%%%%%%%%%%%%%%%%%%%%%%%%%%%%%%%%%%
%% APPENDICES
%%%%%%%%%%%%%%%%%%%%%%%%%%%%%%%%%%%%%%%%%%%%%%%%%%%%%%%%%%%%%
%\appendix
%% ==> Write your text here or include other files.

%\input{parties/annexe} %You need a file 'FileName.tex' for this.

%%%%%%%%%%%%%%%%%%%%%%%%%%%%%%%%%%%%%%%%%%%%%%%%%%%%%%%%%%%%%
%% Aknowledgements
%%%%%%%%%%%%%%%%%%%%%%%%%%%%%%%%%%%%%%%%%%%%%%%%%%%%%%%%%%%%%

\renewcommand{\abstractname}{Aknowledgements}
\begin{abstract}
\noindent I would like to thank Enea Parini and François Hamel for their valued support and useful comments during the elaboration of this document. Let me thank Davide Buoso profusely for useful discussions and for having brought to my attention several interesting references on shape derivatives.
\end{abstract}

%%%%%%%%%%%%%%%%%%%%%%%%%%%%%%%%%%%%%%%%%%%%%%%%%%%%%%%%%%%%%
%% BIBLIOGRAPHY AND OTHER LISTS
%%%%%%%%%%%%%%%%%%%%%%%%%%%%%%%%%%%%%%%%%%%%%%%%%%%%%%%%%%%%%
%% A small distance to the other stuff in the table of contents (toc)
%\addtocontents{toc}{\protect\vspace*{\baselineskip}}

%% The Bibliography
%% ==> You need a file 'literature.bib' for this.
%% ==> You need to run BibTeX for this (Project | Properties... | Uses BibTeX)
%\addcontentsline{toc}{chapter}{Bibliographie} %'Bibliography' into toc
%\nocite{*} %Even non-cited BibTeX-Entries will be shown.
%\bibliographystyle{alpha} %Style of Bibliography: plain / apalike / amsalpha / ...
%\bibliography{../../../../biblio} %You need a file 'literature.bib' for this.

\printbibliography
%\printbibliography[
%heading=bibintoc
%title={Références}
%]

%% The List of Figures
\clearpage
%\addcontentsline{toc}{chapter}{List of Figures}
%\listoffigures

%% The List of Tables
%\clearpage
%\addcontentsline{toc}{chapter}{List of Tables}
%\listoftables

\end{document}